\documentclass[12pt]{amsart}
\usepackage{amsfonts}
\usepackage{graphicx}
\usepackage{amsmath}
\usepackage{amssymb}
\input{amssym.def}
\newtheorem{theorem}{Theorem}
\newtheorem{proposition}{Proposition}
\newtheorem{corollary}{Corollary}

\theoremstyle{remark}

\newcommand{\im}{\text{\rm Im }}

\input epsf
\begin{document}
\title[Harmonic measures]{Harmonic measures of slit sides perpendicular to the domain boundary}
\author[D.~Prokhorov, A.~Zakharov]{Dmitri Prokhorov and Andrey Zakharov}

\address{D.~Prokhorov, A.~Zakharov: Department of Mathematics and Mechanics, Saratov State University, Saratov
410012, Russia} \email{ProkhorovDV@info.sgu.ru} \email{ZakharovAM@info.sgu.ru}

\begin{abstract}
The article is devoted to the geometry of solutions to the chordal L\"owner equation which is based
on comparison of singular solutions and harmonic measures for the sides of a slit in domains
generated by a driving term. It is proved that harmonic measures of two sides of a slit in the
upper half-plane which is perpendicular to the real axis are asymptotically equal to each other.
\end{abstract}
\maketitle

\section{Introduction}

The L\"owner parametric method is one of the powerful tools in geometric function theory. The
famous L\"owner equation was introduced in 1923 \cite{Lowner}. This article is devoted to the
geometry of solutions to the {\it chordal} L\"owner equation which is based on comparison of
singular solutions and harmonic measures for the sides of a slit in domains generated by a driving
term.

The chordal version of the L\"owner equation deals with the upper half-plane $\mathbb H=\{z: \im
z>0\}$, $\mathbb R=\partial \mathbb H$, and functions $f(z,t)$ normalized near infinity by
\begin{equation}
f(z,t)=z+\frac{2t}{z}+O\left(\frac{1}{z^2}\right) \label{exp}
\end{equation}
which solve the chordal L\"owner differential equation
\begin{equation}
\frac{df(z,t)}{dt}=\frac{2}{f(z,t)-\lambda(t)}, \quad f(z,0)\equiv z, \quad t\geq0, \label{Low}
\end{equation}
and map subdomains of $\mathbb H$ onto $\mathbb H$. Here $\lambda(t)$ is a real-valued
continuous driving term.

Let $\gamma(t)$ be a simple continuous curve in $\mathbb H\cup\{0\}$ with $\gamma(0)=0$ and $0\leq
t\leq T$. Then there is a unique map $f(z,t): \mathbb H\setminus\gamma[0,t]\to\mathbb H$ satisfying
the chordal L\"owner equation. It is known that $f(z,t)$ can be extended continuously to $\mathbb
R\cup\gamma(t)$, and $f(\gamma(t),t)=\lambda(t)$. Saying that $\gamma(t)\in C^1$ we mean that the
tangent vector $\gamma'(t)$ exists and varies continuously on $[0,T]$.

The extended function $f(z,t)$ maps $\gamma[0,t]$ onto a segment $I=[f_2(0,t),f_1(0,t)]$ while
$\mathbb R$ is mapped onto $\mathbb R\setminus I$. The two functions $f_1(0,t)$ and $f_2(0,t)$ are
singular solutions to the chordal L\"owner equation which pass through the singular point
$(f(0,0),0)=(0,0)$ of equation (\ref{Low}). The curve $\gamma$ has two sides $\gamma_1$ and
$\gamma_2$ which consist of the same points but define different prime ends, except for its tip.
The two parts $[f_2(0,t),\lambda(t)]$ and $[\lambda(t),f_1(0,t)]$ of segment $I$ are the images of
the two sides of the slit $\gamma[0,t]$ under $f(z,t)$.

The harmonic measures $\omega(f^{-1}(i,t);\gamma_k(t),\mathbb H\setminus\gamma(t))$ of
$\gamma_k(t)$ at $f^{-1}(i,t)$ with respect to $\mathbb H\setminus\gamma(t)$ are defined by the
functions $\omega_k$ which are harmonic on $\mathbb H\setminus\gamma(t)$ and continuously extended
on its closure except for the endpoints of $\gamma$, $\omega_k|_{\gamma_k(t)}=1$,
$\omega_k|_{\mathbb R\cup(\gamma(t)\setminus\gamma_k(t))}=0$, $k=1,2$, see, e.g., [3, \S3.6].
Denote
$$m_k(t):=\omega(f^{-1}(i,t);\gamma_k(t),\mathbb H\setminus\gamma(t)),\;\;\;k=1,2.$$

The main result of the article is given in Theorem 1 which we prove in Section 2.

\begin{theorem}Let $\gamma(t)\in C^1$, $\gamma(0)=0$, $\im\gamma(t)>0$ for $t>0$, be
perpendicular to the real axis $\mathbb R$, and let $f(z,t)$ map $\mathbb H\setminus\gamma[0,t]$
onto $\mathbb H$ and solve the chordal L\"owner differential equation (\ref{Low}) with the diving
function $\lambda(t)$. Then
$$\lim_{t\to+0}\frac{m_1(t)}{m_2(t)}=1.$$
\end{theorem}

Theorem 1 has an interplay with Theorem 2 and its Corollary 1 proved in Section 4.

\begin{theorem}
Let $$\lambda(t)=c\sqrt t+o(\sqrt t),\;\;\;t\to0,$$ be the driving function in the chordal L\"owner
differential equation (\ref{Low}). Then the singular solutions $f_1(0,t)$, $f_2(0,t)$ to(\ref{Low})
satisfy the following conditions
$$\lim_{t\to0}\frac{f_1(0,t)}{\sqrt t}=\frac{c+\sqrt{c^2+16}}{2},\;\;\;
\lim_{t\to0}\frac{f_2(0,t)}{\sqrt t}=\frac{c-\sqrt{c^2+16}}{2}.$$
\end{theorem}

\begin{corollary}
Let $$\lambda(t)=c\sqrt t+o(\sqrt t),\;\;\;t\to0,\;\;\;c\neq0,$$ be the driving function in the
chordal L\"owner differential equation (\ref{Low}), and let the solution $f(z,t)$ to (\ref{Low})
map $\mathbb H\setminus\gamma(t)$ onto $\mathbb H$, where $\gamma(t)$ is a $C^1$ simple curve. Then
$\gamma(t)$ is not perpendicular to the real axis $\mathbb R$ at the origin.
\end{corollary}

The most important argument in the proof of Theorem 1 is the fact \cite{IPV} that for the
arc-length parameter $s$ of the $C^1$-slit which is perpendicular to $\mathbb R$, the function
$s=s(t)$ is expanded as
\begin{equation}
s=s(t)=A\sqrt t+o(\sqrt t),\;\;\;A\neq0,\;\;\;s\to+0. \label{len}
\end{equation}

This result can be compared with the results in \cite{Earle}.

\section{Proof of Theorem 1}

{\it Proof of Theorem 1.} Let $w=f(z,t)$ map $\mathbb H\setminus\gamma(t)$ onto $\mathbb H$, and
let the $C^1$-slit $\gamma$ satisfy the conditions of Theorem 1. For the arc-length parameter $s$,
$\gamma=\gamma(s)$ has the representation
\begin{equation}
\gamma(s)=is+o(s),\;\;\;s\to+0. \label{arc}
\end{equation}
From the other side, the arc-length parameter $s$ is expanded according to (\ref{len}). Substitute
(\ref{len}) in (\ref{arc}) and obtain
$$\gamma(s(t))=iA\sqrt t+o(\sqrt t)=iA\sqrt t+\alpha(t)\sqrt t,\;\;\;\lim_{t\to+0}\alpha(t)=0.$$

The function $$h(w,t)=\sqrt{w^2-A^2t}$$ maps $\mathbb H\cup\mathbb R$ onto $(\mathbb
H\setminus[0,iA\sqrt t])\cup\mathbb R$, $h(0,t)=iA\sqrt t$.

Denote $$z=g_n(w,t):=\sqrt nf^{-1}\left(\frac{w}{\sqrt n},\frac{t}{n}\right),\;\;\;n=1,2,\dots\;.$$
The function $z=f^{-1}(w,t)=g_1(w,t)$ maps $\mathbb H$ onto $\mathbb H\setminus\gamma(t)$,
$g_1(\lambda(t),t)=\gamma(t)$, $\gamma(t):=\gamma(s(t))$. So the functions $g_n(w,t)$ map $\mathbb
H$ onto $\mathbb H\setminus\gamma_n(t)$ where
$$\gamma_n(t)=\sqrt n\,\gamma\left(\frac{t}{n}\right)=
iA\sqrt t+\alpha\left(\frac{t}{n}\right)\sqrt t,\;\;\; g_n\left(\sqrt
n\lambda\left(\frac{t}{n}\right),t\right)=\gamma_n(t).$$

As soon as $$|\gamma_n(t)-iA\sqrt t|=|\alpha\left(\frac{t}{n}\right)|\,\sqrt t$$ tends to 0 as
$n\to\infty$ uniformly with respect to $t\in[0,T]$, the sequence of functions $g_n(w,t)$ converges
to $h(w,t)$ as $n\to\infty$ uniformly on $\mathbb H\cup\mathbb R$ according to the Rad\'o theorem
\cite{Rad} generalized by Markushevich \cite{Mar}, see also [2, p.60].

The uniform convergence of $g_n(w,t)$ to $h(w,t)$ implies convergence of corresponding coefficient
sequences in boundary hydrodynamic normalization (\ref{exp}). Definition of $g_n(w,t)$ gives the
expansion $$g_n(w,t)=w-\frac{2t}{w}+O\left(\frac{1}{w^2}\right),\;\;\;w\to\infty.$$ The function
$h(w,t)$ is expanded as $$h(w,t)=w-\frac{A^2t}{2w}+O\left(\frac{1}{w^2}\right),\;\;\;w\to\infty.$$
Hence $A=2$.

Denote by $\Gamma_1(t)$ the "right" side of the segment $[0,i2\sqrt t]$ and by $\Gamma_2(t)$ the
"left" side of this segment.

Let $\gamma_{1n}(t)$, $\gamma_{2n}(t)$ be the two sides of $\gamma_n(t)$ which are mapped onto the
segments $I_{1n}\subset\mathbb R$, $I_{2n}\subset\mathbb R$ under $g_n^{-1}(z,t)$, respectively.
The uniform convergence of $g_n$ to $h$ implies that $I_{1n}(t)$ tends to $[0,A\sqrt t]$ and
$I_{2n}(t)$ tends to $[-A\sqrt t,0]$ as $n\to\infty$.

Let $\gamma'_{1n}(t)$, $\gamma'_{2n}(t)$ be the two sides of $\gamma({t\over n})$ which are mapped
onto the segments $I'_{1n}\subset\mathbb R$, $I'_{2n}\subset\mathbb R$ under $f(z,{t\over n})$,
respectively. Comparing $I_{kn}$ and $I'_{kn}$, we see that $\text{meas}I_{kn}=\sqrt n\,\text{meas}
I'_{kn}$, $k=1,2$, $n\geq1$.

The harmonic measures $\omega(i;I'_{kn}(t),\mathbb H)$ of $I'_{kn}(t)$ at $i$ with respect to
$\mathbb H$ equal the angle divided over $\pi$ under which the segment $I'_{kn}(t)$ is seen from
the point $i$. Similarly, the harmonic measures $\omega(i;I_{kn}(t),\mathbb H)$ of $I_{kn}(t)$ at
$i$ with respect to $\mathbb H$ equal the angle divided over $\pi$ under which the segment
$I_{kn}(t)$ is seen from the point $i$, $k=1,2$, $n\geq1$, see, e.g., [2, p.334].

Now the following equalities
$$\lim_{t\to+0}\frac{m_1(t)}{m_2(t)}=\lim_{n\to\infty}\frac{m_1({t\over n})}{m_2({t\over n})}=
\lim_{n\to\infty}\frac{\omega(f^{-1}(i,{t\over n}),\gamma_1({t\over n}),\mathbb
H\setminus\gamma({t\over n}))}{\omega(f^{-1}(i,{t\over n}),\gamma_2({t\over n}),\mathbb
H\setminus\gamma({t\over n}))}=$$
$$\lim_{n\to\infty}\frac{\omega(i,I'_{1n}(t),\mathbb H)}{\omega(i,I'_{2n}(t),\mathbb H)}=
\lim_{n\to\infty}\frac{\tan(\pi\omega(i,I'_{1n}(t),\mathbb H))}{\tan(\pi\omega(i,I'_{2n}(t),\mathbb
H))}=\lim_{n\to\infty}\frac{\text{meas}I'_{1n}}{\text{meas}I'_{2n}}=$$
$$\lim_{n\to\infty}\frac{\text{meas}I_{1n}}{\text{meas}I_{2n}}=
\frac{\text{meas}[0,i2\sqrt t]}{\text{meas}[-i2\sqrt t,0]}=1$$ lead to the conclusion desired in
Theorem 1.

This chain contains 8 equality signs. We have to comment almost each of them.

The first equality sign needs a more strict explanation. In order to reduce the limit as $t\to+0$
to the limit as $n\to\infty$ it is necessary to choose an arbitrary sequence $\{j_n\}$ of positive
numbers $j_n$ such that $j_n\to\infty$ as $n\to\infty$. All the arguments for $g_n$ should be
repeated for $g_{j_n}$. We omitted the details and chose the sequence of natural numbers because of
simplicity and evidence of repeating the arguments.

The second step uses the definition of $m_k(t)$. The third step uses the invariance of harmonic
measures with respect to conformal map $f(z,{t\over n})$. The following step uses the limit
property of the ratio of infinitesimal functions. The next step is based on elementary
trigonometric formulas and limit procedures.

The 6-th step takes into account that the segments $I_{kn}$ and $I'_{kn}$ are proportional. The
7-th resulting step appeared because of uniform convergence of $g_n(w,t)$ to $h(w,t)$ which implies
that the pre-images $I_{kn}(t)$ of the sides $\gamma_{kn}(t)$ of $\gamma_n(t)$ under $g_n(w,t)$
tend to the corresponding pre-images of the sides $\Gamma_k(t)$ of the segment $[0,i2\sqrt t]$
under $h(w,t)$, $k=1,2$. The final step is clear.

This completes the proof.

Theorem 1 can be generalized for $C^1$-slits $\gamma(t)\subset\mathbb H$, $\gamma(0)=0$, which are
tangential at the origin to the straight line under the angle ${\pi\over2}(1-c)$ to $\mathbb R$,
$-1<c<1$, provided the asymptotic relation (\ref{len}) is valid.

\begin{proposition}
Let $\gamma(t)\in C^1$, $\gamma(0)=0$, $\im\gamma(t)>0$ for $t>0$, be tangential at the origin to
the straight line under the angle ${\pi\over2}(1-c)$, $-1<c<1$, to the real axis $\mathbb R$, and
let $f(z,t)$ map $\mathbb H\setminus\gamma[0,t]$ onto $\mathbb H$ and solve the chordal L\"owner
differential equation (\ref{Low}). Then
$$\lim_{t\to+0}\frac{M_1(t)}{M_2(t)}=\frac{1-c}{1+c},$$ where
$$M_k(t):=\omega(f^{-1}(i,t);\gamma_k(t),\mathbb H\setminus\gamma(t)),\;\;\;k=1,2,$$ $\gamma_1(t)$
and $\gamma_2(t)$ are the two sides of $\gamma(t)$, provided the asymptotic relation (\ref{len}) is
valid.
\end{proposition}

We omit the proof which can follow the steps of Theorem 1 where the slit perpendicular to $\mathbb
R$ was generated by the trivial explicit map $h(w,t)$. Instead, the straight line under the angle
${\pi\over2}(1-c)$ is generated by the implicitly given map in \cite{Kager} which solves the
L\"owner differential equation (\ref{Low}) with the driving function $\lambda(t)=c'\sqrt t$.
Therefore the proof is more complicated technically but does not contain any new ideas.

\section{Driving terms with higher Lipschitz orders}

\begin{theorem}
Let $f(z,t)$ be a solution to the chordal L\"owner equation (\ref{Low}) with the driving function
$$\lambda(t)=At^{\alpha}+o(t^{\alpha}),\;\;\;t\to0,\;\;\;\alpha>\frac{1}{2},\;\;\;A\neq0.$$
Then the singular solutions $f_1(0,t)$ and $f_2(0,t)$ to (\ref{Low}) satisfy the following
relations
$$f_1(0,t)=2\sqrt t+o(\sqrt t),\;\;\;f_2(0,t)=-2\sqrt t+o(\sqrt t),\;\;\;t\to0.$$
\end{theorem}

\begin{proof}
For all $t>0$, $$f_2(0,t)<\lambda(t)<f_1(0,t)\;\;\;\text{and}\;\;\; f_2(0,t)<0<f_1(0,t).$$

For given $\epsilon>0$, $|\lambda(t)|<\epsilon\sqrt t$ for $t>0$ small enough. Therefore, for such
$t>0$,
\begin{equation}
\frac{df_1(0,t)}{dt}=\frac{2}{f_1(0,t)-\lambda(t)}>\frac{2}{f_1(0,t)+\epsilon\sqrt t}. \label{5}
\end{equation}
Denote $$g(t)=\frac{f_1(0,t)}{\sqrt t}$$ and obtain from (\ref{5})) that
\begin{equation}
\frac{dg}{dt}>\frac{1}{t}\;\frac{4-g^2(t)-\epsilon g(t)}{2(g(t)+\epsilon)},\;\;\; 0<t<T(\epsilon).
\label{6}
\end{equation}

The polynomial $4-g^2-\epsilon g$ has two roots $g_2(\epsilon)<0$ and
$$g_1(\epsilon)=\frac{\sqrt{\epsilon^2+16}-\epsilon}{2},\;\;\;0<g_1(\epsilon)<2.$$

Suppose that there exists $t_0\in(0,T(\epsilon))$ such that $g(t_0)<g_1(\epsilon)$. It follows from
(\ref{6}) that $g'(t_0)>0$, and $g(t)$ decreases together with $t$ varying from $t_0$ to 0. This
implies that $$\frac{4-g^2(t)-\epsilon g(t)}{2(g(t)+\epsilon)}>c>0,\;\;\;0<t\leq t_0,$$ and
inequality (\ref{6}) reduces to
\begin{equation}
\frac{dg}{dt}>\frac{c}{t},\;\;\;0<t\leq t_0. \label{7}
\end{equation}
Integrating (\ref{7}) from $\delta>0$ to $t_0$, we obtain
$$g(t_0)>g(\delta)+c\log{\frac{t_0}{\delta}}>c\log{\frac{t_0}{\delta}},$$ which contradicts
the condition $g(t_0)<g_1(\epsilon)$ for $\delta$ small enough.

So, for all $t\in(0,T(\epsilon))$, $g(t)\geq g_1(\epsilon)$. Going back to $f_1(0,t)$, we see that
\begin{equation}
f_1(0,t)\geq g_1(\epsilon)\sqrt t,\;\;\;t\in(0,T(\epsilon)). \label{8}
\end{equation}
Substitute (\ref{8}) in (\ref{Low}) and obtain that
$$\frac{df_1(0,t)}{dt}=\frac{2}{f_1(0,t)-\lambda(t)}\leq\frac{2}{g_1(\epsilon)\sqrt
t-\epsilon\sqrt t},\;\;\;0<t<T(\epsilon).$$ Integrate this inequality from 0 to $t$ and obtain
that
\begin{equation}
f_1(0,t)\leq\frac{4}{g_1(\epsilon)-\epsilon}\sqrt t,\;\;\;0<t<T(\epsilon). \label{9}
\end{equation}

Inequalities (\ref{8}) and (\ref{9}) mean together that $$\lambda(t)=o(f_1(0,t)),\;\;\;t\to0,$$ and
$$\frac{df_1(0,t)}{dt}=\frac{2}{f_1(0,t)+o(f_1(0,t))}.$$ This leads to the first statement of
Theorem 3 for $f_1(0,t)$.

The second statement of Theorem 3 for $f_2(0,t)$ is proved similarly.

\end{proof}

\section{Proof of Theorem 2 and Corollary 1}

{\it Proof of Theorem 2.} For given $\epsilon'>0$, $|\lambda(t)-c\sqrt t|<\epsilon'\sqrt t$ for
$t>0$ small enough. Therefore, for such $t>0$,
\begin{equation}
\frac{df_1(0,t)}{dt}=\frac{2}{f_1(0,t)-\lambda(t)}>\frac{2}{f_1(0,t)-(c-\epsilon')\sqrt t}.
\label{10}
\end{equation}
Denote $$g(t)=\frac{f_1(0,t)}{\sqrt t}$$ and obtain from (\ref{10}) that
\begin{equation}
\frac{dg}{dt}>\frac{1}{t}\;\frac{4-g^2(t)+(c-\epsilon')g(t)}{2(g(t)-c+\epsilon')},\;\;\;
0<t<T(\epsilon'). \label{11}
\end{equation}

The polynomial $4-g^2+(c-\epsilon')g$ has two roots $g_2(\epsilon')<0$ and
$$g_1(\epsilon')=\frac{\sqrt{(c-\epsilon')^2+16}+c-\epsilon'}{2},\;\;\;0<g_1(\epsilon')<2.$$

Suppose that there exists $t_0\in(0,T(\epsilon'))$ such that $g(t_0)<g_1(\epsilon')$. It follows
from (\ref{11}) that $g'(t_0)>0$, and $g(t)$ decreases together with $t$ varying from $t_0$ to 0.
This implies that
$$\frac{4-g^2(t)+(c-\epsilon')g(t)}{2(g(t)+\epsilon')}>p>0,\;\;\;0<t\leq t_0,$$ and inequality
(\ref{11}) reduces to
\begin{equation}
\frac{dg}{dt}>\frac{p}{t},\;\;\;0<t\leq t_0. \label{12}
\end{equation}
Integrating (\ref{12}) from $\delta>0$ to $t_0$, we obtain
$$g(t_0)>g(\delta)+p\log{\frac{t_0}{\delta}}>p\log{\frac{t_0}{\delta}},$$ which contradicts
the condition $g(t_0)<g_1(\epsilon')$ for $\delta$ small enough.

So, for all $t\in(0,T(\epsilon'))$, $g(t)\geq g_1(\epsilon')$. Going back to $f_1(0,t)$, we see
that
\begin{equation}
f_1(0,t)\geq g_1(\epsilon')\sqrt t,\;\;\;t\in(0,T(\epsilon')). \label{13}
\end{equation}
Substitute (\ref{13}) in (\ref{Low}) and obtain that
$$\frac{df_1(0,t)}{dt}=\frac{2}{f_1(0,t)-\lambda(t)}\leq\frac{2}{g_1(\epsilon')\sqrt
t-(c+\epsilon')\sqrt t},\;\;\;0<t<T(\epsilon').$$ Integrate this inequality from 0 to $t$ and
obtain that
\begin{equation}
f_1(0,t)\leq\frac{4}{g_1(\epsilon')-(c+\epsilon')}\sqrt t,\;\;\;t\in(0,T(\epsilon')). \label{14}
\end{equation}

If $c=0$, then inequalities (\ref{13}) and (\ref{14}) prove the first statement of Theorem 2 for
$f_1(0,t)$.

Let $c\neq0$. Write two inequalities (\ref{13}) and (\ref{14}) in the form
\begin{equation}
k_1'\sqrt t<f_1(0,t)<k_1''\sqrt t,\;\;\;t\in(0,T(\epsilon')). \label{15}
\end{equation}

Show that inequalities (\ref{15}) admit a recurrent improvement converging to the desired point.
Indeed, substitute the left inequality (\ref{15}) in (\ref{Low}) and obtain
$$\frac{df_1(0,t)}{dt}=\frac{2}{f_1(0,t)-c\sqrt t+o(\sqrt t)}<\frac{2}{(k_1'-c)\sqrt t+o(\sqrt
t)}$$ which gives after integration the improved right inequality (\ref{15}) $$f_1(0,t)<k_2''\sqrt
t+o(\sqrt t),\;\;\;t\to0,\;\;\;k_2''=\frac{4}{k_1'-c}.$$ Now substitute this inequality in
(\ref{Low}) and obtain $$\frac{df_1(0,t)}{dt}>\frac{2}{(k_2''-c)\sqrt t+o(\sqrt t)}$$ which gives
after integration the improved left inequality (\ref{15}) $$f_1(0,t)>k_2'\sqrt t+o(\sqrt
t),\;\;\;t\to0,\;\;\;k_2'=\frac{4}{k_2''-c}.$$

Repeat the procedure and obtain at the $n$-th step
\begin{equation}
k_n'\sqrt t+o(\sqrt t)<f_1(0,t)<k_n''\sqrt t+o(\sqrt t), \label{16}
\end{equation}
where $$k_n'=\frac{4}{\frac{4}{k_{n-1}'-c}-c},\;\;\;k_n''=\frac{4}{\frac{4}{k_{n-1}''-c}-c}.$$
Inequalities $k_{n-1}'>c$, $k_{n-1}''>c$, $4>c(k_{n-1}'-c)$, $4>c(k_{n-1}''-c)$ are verified
by elementary sources. Both sequences $\{k_n'\}$ and $\{k_n''\}$ converge monotonically to
$(c+\sqrt{c^2+16})/2$. To every $\epsilon>0$ there exists $n\in\mathbb N$ such that
$$0<\frac{c+\sqrt{c^2+16}}{2}-k_n'<\frac{\epsilon}{2},\;\;\;
0<k_n''-\frac{c+\sqrt{c^2+16}}{2}<\frac{\epsilon}{2},$$ and $$\frac{o(\sqrt t)}{\sqrt
t}<\frac{\epsilon}{2}$$ for both terms $o(\sqrt t)$ in (\ref{16}). This proves the first statement
of Theorem 2 for $f_1(0,t)$.

The second statement of Theorem 2 for $f_2(0,t)$ is proved similarly. \vskip2mm

{\it Proof of Corollary 1.} According to Theorem 1, for $C^1$ curves $\gamma(t)$ which are
perpendicular to $\mathbb R$, $f_1(0,t)$ and $f_2(0,t)$ have the same main term in the asymptotic
expansion. From the other side, according to Theorem 2, $f_1(0,t)$ and $f_2(0,t)$ differ by their
main asymptotic terms. This proves Corollary 1.

\end{document}